\documentclass{amsart}
\author{Dilip Raghavan}
\thanks{First author partially supported by Grants-in-Aid for Scientific Researchfor JSPS
Fellows No.\ 23$\cdot$01017}
\address{Graduate school of system informatics\\
Kobe University\\
Kobe 657-8501, Japan.}
\email{raghavan@math.toronto.edu}
\urladdr{http://www.math.toronto.edu/raghavan}
\author{Saharon Shelah}
\thanks{Research partially supported by NSF grant DMS 1101597, and by
the United States-Israel Binational Science Foundation (Grant no. 2006108). Publication 991.}
\date{October 12, 2011}
\title{Comparing the closed almost disjointness and dominating numbers}
\usepackage{amssymb, amsmath, amsthm, mathrsfs, enumerate, amsfonts, latexsym, bbm}
\def\polhk#1{\setbox0=\hbox{#1}{\ooalign{\hidewidth
    \lower1.5ex\hbox{`}\hidewidth\crcr\unhbox0}}}
\newtheorem{Theorem}{Theorem}

\newtheorem{Lemma}[Theorem]{Lemma}

\theoremstyle{definition}
\newtheorem{Question}[Theorem]{Question}
\theoremstyle{definition}
\newtheorem{Def}[Theorem]{Definition}

\theoremstyle{remark}

\newcommand{\restrict}{\upharpoonright}
\newcommand{\forallbutfin}{{\forall}^{\infty}}

\renewcommand{\c}{\mathfrak{c}}
\renewcommand{\b}{\mathfrak{b}}
\newcommand{\hh}{\mathbf{h}}
\renewcommand{\d}{\mathfrak{d}}
\newcommand{\s}{\mathfrak{s}}
\newcommand{\h}{\mathfrak{h}}

\renewcommand{\a}{{\mathfrak{a}}}
\newcommand{\ac}{{\mathfrak{a}}_{closed}}
\renewcommand{\[}{\left[}
\renewcommand{\]}{\right]}
\renewcommand{\P}{\mathbb{P}}

\newcommand{\lc}{\left|}
\newcommand{\rc}{\right|}

\newcommand\CH{\mathrm{CH}}   
\newcommand{\BS}{{\omega}^{\omega}}

\DeclareMathOperator{\rk}{rk}

\DeclareMathOperator{\dom}{dom}
\DeclareMathOperator{\ran}{ran}

\DeclareMathOperator{\cf}{cf}

\newcommand{\A}{{\mathscr{A}}}

\newcommand{\cube}{{\[\omega\]}^{\omega}}
\newcommand{\fin}{{{\[\omega\]}^{< \omega}}}

\newcommand{\V}{{\mathbf{V}}}
\newcommand{\LL}{{\mathbf{L}}}
\newcommand{\VG}{{{\mathbf{V}}[G]}}

\begin{document}
\begin{abstract}
	We prove that if there is a dominating family of size ${\aleph}_{1}$, then there is are ${\aleph}_{1}$ many compact subsets of $\BS$ whose union is a maximal almost disjoint family of functions that is also maximal with respect to infinite partial functions.
\end{abstract}
\maketitle
\section{Introduction} \label{sec:intro}
	Recall that two infinite subsets $a$ and $b$ of $\omega$ are \emph{almost disjoint or a.d.\@} if $a \cap b$ is finite. A family $\A$ of infinite subsets of $\omega$ is said to be \emph{almost disjoint or a.d.\@ in $\cube$} if its members are pairwise almost disjoint. A \emph{Maximal Almost Disjoint family, or MAD family in $\cube$} is an infinite a.d.\ family in $\cube$ that is not properly contained in a larger a.d.\ family. 

	Two functions $f$ and $g$ in $\BS$ are said to be \emph{almost disjoint or a.d.\@} if they agree in only finitely many places. We say that a family $\A \subset \BS$ is \emph{a.d.\ in $\BS$} if its members are pairwise a.d.\@, and we say that an a.d.\ family $\A \subset \BS$ is \emph{MAD in $\BS$} if $\forall f \in \BS \exists h \in \A \[\lc f \cap h \rc = {\aleph}_{0}\]$. Identifying functions with their graphs, every a.d.\ family in $\BS$ is also an a.d.\ family in ${\[\omega \times \omega \]}^{\omega}$; however, it is never MAD in ${\[\omega \times \omega\]}^{\omega}$ because any function is a.d.\ from the vertical columns of $\omega \times \omega$. MAD families in $\BS$ that become MAD in ${\[\omega \times \omega \]}^{\omega}$ when the vertical columns of $\omega \times \omega$ are thrown in were considered by Van Douwen.

	We say that $p \subset \omega \times \omega$ is an \emph{infinite partial function} if it is a function from some infinite set $A \subset \omega$ to $\omega$. An a.d.\ family $\A \subset \BS$ is said to be \emph{Van Douwen} if for any infinite partial function $p$ there is $h \in \A$ such that $\lc h \cap p \rc = {\aleph}_{0}$. $\A$ is Van Douwen iff $\A \cup \{{c}_{n}: n \in \omega\}$ is a MAD family in ${\[\omega \times \omega\]}^{\omega}$, where ${c}_{n}$ is the $n$th vertical column of $\omega \times \omega$. The first author showed in \cite{vmad} that Van Douwen MAD families always exist.

	Recall that $\b$ is the least size of an unbounded family in $\BS$, $\d$ is the least size of a dominating family in $\BS$, and $\a$ is the least size of a MAD family in $\cube$. It is well known that $\b \leq \a$. Whether $\a$ could consistently be larger than $\d$ was an open question for a long time, until Shelah achieved a breakthrough in \cite{sh700} by producing a model where $\d = {\aleph}_{2}$ and $\a = {\aleph}_{3}$. However, it is not known whether $\a$ can be larger than $\d$ when $\d = {\aleph}_{1}$; this is one of the few major remaining open problems in the theory of cardinal invariants posed during the earliest days of the subject (see \cite{vdsmall} and \cite{opentopII}). In this note we take a small step towards resolving this question by showing that if $\d = {\aleph}_{1}$, then there is a MAD family in $\cube$ which is the union of ${\aleph}_{1}$ compact subsets of $\cube$. More precisely, we will establish the following:
\begin{Theorem} \label{thm:main}
	Assume $\d = {\aleph}_{1}$. Then there exist ${\aleph}_{1}$ compact subsets of $\BS$ whose union is a Van Douwen MAD family.
\end{Theorem}
The cardinal invariant $\ac$ was recently introduced and studied by Brendle and Khomskii~\cite{perfectmad} in connection with the possible descriptive complexities of MAD families in certain forcing extensions of $\LL$.
\begin{Def} \label{def:ac}
	$\ac$ is the least $\kappa$ such that there are $\kappa$ closed subsets of $\cube$ whose union is a MAD family in $\cube$. 
\end{Def}
Obviously, $\ac \leq \a$. Brendle and Khomskii showed in \cite{perfectmad} that $\ac$ behaves differently from $\a$ by producing a model where $\ac = {\aleph}_{1} < {\aleph}_{2} = \b$. They asked whether $\s = {\aleph}_{1}$ implies that $\ac = {\aleph}_{1}$. As $\s \leq \d$, our result in this paper provides a partial positive answer to their question.
\section{The construction} \label{sec:proof}
Assume $\d = \aleph_1$ in this section. We will build $\aleph_1$ many compact subsets of $\BS$ whose union is a Van Douwen MAD family. To this end, we will construct a sequence $\langle {T}_{\alpha}: \alpha < \omega_1 \rangle$ of finitely branching subtrees of ${\omega}^{< \omega}$ such that ${\bigcup}_{\alpha < \omega_1}{\[{T}_{\alpha}\]}$ has the required properties. Henceforth, $T \subset {\omega}^{< \omega}$ will mean $T$ is a \emph{subtree} of ${\omega}^{< \omega}$. 
\begin{Def} \label{def:rk}
	Let $T \subset {\omega}^{< \omega}$. Let $A \in \cube$ and $p: A \rightarrow \omega$. For any ordinal $\xi$ and $\sigma \in T$ define ${\rk}_{T, p}(\sigma) \geq \xi$ to mean 
	\begin{align*}
		\forall \zeta < \xi \exists \tau \in T \exists l \in A\[\tau \supset \sigma \wedge \lc \sigma \rc \leq l < \lc \tau \rc \wedge \tau(l) = p(l) \wedge {\rk}_{T, p}(\tau) \geq \zeta\].
	\end{align*}
\end{Def}
Note that if $\eta \leq \xi$ and ${\rk}_{T, p}(\sigma) \geq \xi$, then ${\rk}_{T, p}(\sigma) \geq \eta$, and that for a limit ordinal $\xi$, if $\forall \zeta < \xi\[{\rk}_{T, p}(\sigma) \geq \zeta\]$, then ${\rk}_{T, p}(\sigma) \geq \xi$. Also, for any $\sigma, \tau \in T$, if $\sigma \subset \tau$ and ${\rk}_{T, p}(\tau) \geq \xi$, then ${\rk}_{T, p}(\sigma) \geq \xi$. Moreover, if ${\rk}_{T, p}(\sigma) \not\geq \xi$ and if $\tau \in T$ and $l \in A$ are such that $\tau \supset \sigma$, $\lc \sigma \rc \leq l < \lc \tau \rc$, and $p(l) = \tau(l)$, then there is $\zeta < \xi$ such that ${\rk}_{T, p}(\tau) \not\geq \zeta$. Therefore, if there is $f \in \[T\]$ with $\lc f \cap p \rc = {\aleph}_{0}$, and $\sigma \subset f$ and there is some ordinal $\xi$ such that ${\rk}_{T, p}(\sigma) \not\geq \xi$, then is some $\sigma \subset \tau \subset f$ and some ordinal $\zeta < \xi$ such that ${\rk}_{T, p}(\tau) \not\geq \zeta$, thus allowing us to construct an infinite, strictly descending sequence of ordinals. So if $f \in \[T\]$ with $\lc f \cap p \rc = {\aleph}_{0}$, then for any $\sigma \subset f$ and any ordinal $\xi$, ${\rk}_{T, p}(\sigma) \geq \xi$. On the other hand, suppose that $\sigma \in T$ with ${\rk}_{T, p}(\sigma) \geq {\omega}_{1}$. Then there is $\tau \in T$ with $\tau \supset \sigma$ and $l \in A$ such that $\lc \sigma \rc \leq l < \lc \tau \rc$, $p(l) = \tau(l)$, and ${\rk}_{T, p}(\tau) \geq {\omega}_{1}$, allowing us to construct $f \in \[T\]$ with $\sigma \subset f$ such that $\lc f \cap p \rc = {\aleph}_{0}$.
\begin{Def} \label{def:rk1}
	Suppose $T \subset {\omega}^{< \omega}$, $A \in \cube$, and $p: A \rightarrow \omega$. Assume that $p$ is a.d.\ from each $f \in \[T\]$. Then define ${H}_{T, p}: T \rightarrow {\omega}_{1}$ by ${H}_{T, p}(\sigma) = \min\{\xi: {\rk}_{T, p}(\sigma) \not\geq \xi + 1\}$.
\end{Def}
Note the following features of this definition
	\begin{enumerate}
		\item[(${\ast}_{1}$)]
			$\forall \sigma, \tau \in T\[\sigma \subset \tau \implies {H}_{T, p}(\sigma) \geq {H}_{T, p}(\tau)\]$
		\item[(${\ast}_{2}$)]
			for all $\sigma, \tau \in T$ with $\sigma \subset \tau$, if there exists $l \in A$ such that $\lc \sigma \rc \leq l < \lc \tau \rc$ and $p(l) = \tau(l)$, then ${H}_{T, p}(\tau) < {H}_{T, p}(\sigma)$.
	\end{enumerate}
On the other hand, notice that if there is a function $H: T \rightarrow {\omega}_{1}$ such that (${\ast}_{1}$) and (${\ast}_{2}$) hold when ${H}_{T, p}$ is replaced with $H$, then $p$ must be a.d.\ from $[T]$.
\begin{Def} \label{def:interval}
	$I$ is said to be an \emph{interval partition} if $I = \langle {i}_{n}: n \in \omega \rangle$, where ${i}_{0} = 0$, and $\forall n \in \omega \[{i}_{n} < {i}_{n + 1}\]$. For $n \in \omega$, ${I}_{n}$ denotes the interval $[{i}_{n}, {i}_{n + 1})$. 

	Given two interval partitions $I$ and $J$, we say that $I$ \emph{dominates} $J$ and write $J {\leq}^{\ast} I$ if $\forallbutfin n \in \omega \exists k \in \omega \[{J}_{k} \subset {I}_{n} \]$.   
\end{Def}
It is well known that $\d$ is also the size of the smallest family of interval partitions dominating any interval partition. So fix a sequence $\langle {I}^{\alpha}: \alpha < {\omega}_{1} \rangle$ of interval partitions such that
	\begin{enumerate}
		\item
			$\forall \alpha \leq \beta < {\omega}_{1} \[{I}^{\alpha} {\leq}^{\ast} {I}^{\beta} \]$
		\item
			for any interval partition $J$, there exists $\alpha < {\omega}_{1}$ such that $J {\leq}^{\ast} {
I}^{\alpha}$.
	\end{enumerate}
Fix an $\omega_1$-scale $\langle {f}_{\alpha}: \alpha < \omega_1 \rangle$ such that $\forall \alpha < \omega_1 \forall n \in \omega \[f_\alpha(n) <  f_\alpha(n + 1) \]$. For each $\alpha \geq 1$, define ${e}_{\alpha}$ and $g_\alpha$ by induction on $\alpha$ as follows. If $\alpha$ is a successor, then $e_\alpha: \omega \rightarrow \alpha$ is any onto function, and ${g}_{\alpha} = f_\alpha$. If $\alpha$ is a limit, then let $\{e_n: n \in \omega\}$ enumerate $\{{e}_{\xi}: \xi < \alpha\}$. Now, define ${e}_{\alpha}: \omega \rightarrow \alpha$ and $g_\alpha \in \BS$ such that
		\begin{enumerate}
			\item[(3)]
				$\forall n \in \omega \[g_\alpha(n) \leq g_\alpha(n + 1)\]$
			\item[(4)]
				$\forall n \in \omega \forall i \leq n \forall j \leq f_\alpha(n) \exists k < g_\alpha(n)\[e_\alpha(k) = e_i(j)\]$.
		\end{enumerate} 
Observe that such an $e_\alpha$ must be a surjection. For each $n \in \omega$, put ${w}_{\alpha}(n) = \{{e}_{\alpha}(i): i \leq {g}_{\alpha}(n)\}$.

	Now fix $\alpha < \omega_1$ and assume that ${T}_{\epsilon} \subset {\omega}^{< \omega}$ has been defined for each $\epsilon < \alpha$ such that each ${T}_{\epsilon}$ is finitely branching and ${\bigcup}_{\epsilon < \alpha}{\[{T}_{\epsilon}\]}$ is an a.d.\ family in $\BS$. Let $\langle {\epsilon}_{n}: n \in \omega \rangle$ enumerate $\alpha$, possibly with repetitions. For a tree $T \subset {\omega}^{< \omega}$ and $l \in \omega$, $T \restrict l$ denotes $\{\sigma \in T: \lc \sigma \rc \leq l \}$, and $T(l)$ denotes $\{\sigma \in T: \lc \sigma \rc = l \}$. We will define a sequence of natural numbers $0 = {l}_{0} < {l}_{1} < \dotsb$ and determine ${T}_{\alpha} \restrict {l}_{n}$ by induction on $n$. ${T}_{\alpha} \restrict {l}_{0} = \{0\}$. Assume that ${l}_{n}$ and ${T}_{\alpha} \restrict {l}_{n}$ are given. Suppose also that we are given a sequence of natural numbers $\langle {k}_{i}: i < n \rangle$ such that
	\begin{enumerate}
		\item[(5)]
			$\forall i < i + 1 < n \[{k}_{i} < {k}_{i + 1}\]$
		\item[(6)]
			${I}^{\alpha}_{{k}_{i}} \subset [0, {l}_{n})$.
	\end{enumerate}
Let ${\sigma}^{\ast}$ denote the member of ${T}_{\alpha}({l}_{n})$ that is right most with respect to the lexicographical ordering on ${\omega}^{{l}_{n}}$. Suppose we are also given ${L}_{n}: {T}_{\alpha}({l}_{n}) \setminus \{{\sigma}^{\ast}\} \rightarrow {W}_{n}$, an injection. Here ${W}_{n}$ is the set of all pairs $\langle {p}_{0}, \bar{h} \rangle$ such that
	\begin{enumerate}
		\item[(7)]
			there are $s \in \fin$, and numbers ${i}_{0} < {j}_{0} \leq n$ such that
				\begin{enumerate}
					\item[(a)]
						$s \subset {\bigcup}_{i \in [{i}_{0}, {j}_{0})}{{I}^{\alpha}_{{k}_{i}}}$
					\item[(b)]
						for each $i \in [{i}_{0}, {j}_{0})$, $\lc s \cap {I}^{\alpha}_{{k}_{i}}\rc = 1$
					\item[(c)]
						${p}_{0}: s \rightarrow \omega$ such that $\forall m \in s \[{p}_{0}(m) \leq {f}_{\alpha}(m)\]$
				\end{enumerate} 
		\item[(8)]
			There is ${j}_{1} < n$ such that $\bar{h} = \langle {h}_{{\epsilon}_{i}}: i \leq {j}_{1}\rangle$ (if $\alpha = 0$, this means that $\bar{h} = 0$). For each $i \leq {j}_{1}$, ${h}_{{\epsilon}_{i}}: {T}_{{\epsilon}_{i}} \restrict \max{(s)} + 1 \rightarrow {w}_{\alpha}(\max{(s)} + 1)$ such that (${\ast}_{1}$) and (${\ast}_{2}$) hold when $T$ is replaced there with ${T}_{{\epsilon}_{i}} \restrict \max{(s)} + 1$, ${H}_{T, p}$ is replaced with ${h}_{{\epsilon}_{i}}$, $A$ with $s$, and $p$ with ${p}_{0}$. 	
	\end{enumerate}
Assume that for each $i < n$, we are also given ${\sigma}_{i} \in {T}_{\alpha}({l}_{i})$, which we will call \emph{the active node at stage $i$}. Note that ${T}_{\alpha}({l}_{0}) = \{0\}$, and so ${\sigma}_{0} = 0$. For each $\sigma \in {T}_{\alpha}({l}_{n})$, let $\Delta(\sigma) = \max\left( \{ 0 \} \cup \{i < n: {\sigma}_{i} = \sigma \restrict {l}_{i}\} \right)$. For, $\sigma, \tau \in {T}_{\alpha}({l}_{n})$, say $\sigma \vartriangleleft \tau$ if either $\Delta(\sigma) < \Delta(\tau)$ or $\Delta(\sigma) = \Delta(\tau)$ and $\sigma$ is to the left of $\tau$ in the lexicographic ordering on ${\omega}^{{l}_{n}}$. Let ${\sigma}_{n}$ be the $\vartriangleleft$-minimal member of ${T}_{\alpha}({l}_{n})$. ${\sigma}_{n}$ will be active at stage $n$. The meaning of this is that none of the other nodes in ${T}_{\alpha}({l}_{n})$ will be allowed to branch at stage $n$. Choose ${k}_{n}$ greater than all ${k}_{i}$ for $i < n$ such that ${I}^{\alpha}_{{k}_{n}} \subset [{l}_{n}, \infty)$. Let ${V}_{n}$ be the set of all pairs $\langle {p}_{1}, \bar{\hh} \rangle$ such that
	\begin{enumerate}
		\item[(9)]
			there exist $s$ and a natural number ${i}_{1} \leq n$ such that
				\begin{enumerate}
					\item[(a)]
						$s \subset {\bigcup}_{i \in [{i}_{1}, n + 1)}{{I}^{\alpha}_{{k}_{i}}}$
					\item[(b)]
						for each $i \in [{i}_{1}, n + 1)$, $\lc s \cap {I}^{\alpha}_{{k}_{i}} \rc = 1$
					\item[(c)]
						${p}_{1}: s \rightarrow \omega$ such that $\forall m \in s \[{p}_{1}(m) \leq {f}_{\alpha}(m)\]$
				\end{enumerate}
		\item[(10)]
			There is ${j}_{2} \leq n$ such that $\bar{\hh} = \langle {\hh}_{{\epsilon}_{i}}: i \leq {j}_{2} \rangle$. For each $i \leq {j}_{2}$, ${\hh}_{{\epsilon}_{i}}: {T}_{{\epsilon}_{i}} \restrict \max{(s)} + 1 \rightarrow {w}_{\alpha}(\max{(s)} + 1)$ such that (${\ast}_{1}$) and (${\ast}_{2}$) are satisfied when $T$ is replaced with ${T}_{{\epsilon}_{i}} \restrict \max{(s)} + 1$, ${H}_{T, p}$ is replaced with ${\hh}_{{\epsilon}_{i}}$, $A$ with $s$, and $p$ with ${p}_{1}$.
	\end{enumerate}
Note that ${V}_{n}$ is always finite. Now, the construction splits into two cases.

Case I: ${\sigma}_{n} \neq {\sigma}^{\ast}$. Put $\langle {p}_{0}, \bar{h} \rangle = {L}_{n}({\sigma}_{n})$. Let ${i}_{0}< n$ be as in (7) above, and let ${j}_{1} < n $ be as in (8). Let
	\begin{align*} 
		{U}_{n} = \left\{ \langle {p}_{1}, \bar{\hh} \rangle \in {V}_{n}: {p}_{0} \subset {p}_{1} \wedge {i}_{0} = {i}_{1} \wedge {j}_{1} < {j}_{2} \wedge \forall i \leq {j}_{1} \[{\hh}_{{\epsilon}_{i}} \restrict \dom ({h}_{{\epsilon}_{i}}) = {h}_{{\epsilon}_{i}} \] \right\}.
	\end{align*}
Here ${i}_{1}$ is as in (9), and ${j}_{2}$ is as in (10) with respect to $\langle {p}_{1}, \bar{\hh} \rangle$. Now choose ${l}_{n + 1} > {l}_{n}$ large enough so that ${I}^{\alpha}_{{k}_{n}} \subset [{l}_{n}, {l}_{n + 1})$ and so that it is possible to pick $\{{\tau}_{x}: x \in {U}_{n}\} \subset {\omega}^{{l}_{n + 1}}$ and $\{{\tau}_{\sigma}: \sigma \in {T}_{\alpha}({l}_{n})\} \subset {\omega}^{{l}_{n + 1}}$ such that the following conditions are satisfied.
	\begin{enumerate}
		\item[(11)]
			for each $x \in {U}_{n}$, ${\tau}_{x} \supset {\sigma}_{n}$, and for each $\sigma \in {T}_{\alpha}({l}_{n})$, ${\tau}_{\sigma} \supset \sigma$
		\item[(12)]
			for each $x, y \in {U}_{n}$, if $x \neq y$, then there exists $m \in [{l}_{n}, {l}_{n + 1})$ such that ${\tau}_{x}(m) \neq {\tau}_{y}(m)$. For each $x \in {U}_{n}$, there exists $m \in [{l}_{n}, {l}_{n + 1})$ such that ${\tau}_{x}(m) \neq {\tau}_{{\sigma}_{n}}(m)$. For $x = \langle {p}_{1}, \bar{\hh} \rangle \in {U}_{n}$, if $\{{i}^{\ast}\} = \dom{({p}_{1})} \cap {I}^{\alpha}_{{k}_{n}}$, then ${p}_{1}({i}^{\ast}) = {\tau}_{x}({i}^{\ast})$.		
		\item[(13)]
			for each $x \in {U}_{n}$ and $\sigma \in {T}_{\alpha}({l}_{n})$, $\forall m \in [{l}_{n}, {l}_{n + 1})\[{\tau}_{x}(m) \neq {\tau}_{\sigma}(m)\]$. For $\sigma, \eta \in {T}_{\alpha}({l}_{n})$, if $\sigma \neq \eta$, then $\forall m \in [{l}_{n}, {l}_{n + 1})\[{\tau}_{\sigma}(m) \neq {\tau}_{\eta}(m)\]$.
		\item[(14)]
			for each $i \leq n$, $\tau \in {T}_{{\epsilon}_{i}}({l}_{n + 1})$, $\sigma \in {T}_{\alpha}({l}_{n})$ and $m \in [{l}_{n}, {l}_{n + 1})$, $\tau(m) \neq {\tau}_{\sigma}(m)$. For each $x \in {U}_{n}$, $i \leq {j}_{2}$, $\tau \in {T}_{{\epsilon}_{i}}({l}_{n + 1})$ and $m \in [{l}_{n}, {l}_{n + 1})$, if ${\tau}_{x}(m) = \tau(m)$, then $m \in \dom{({p}_{1})}$ and ${p}_{1}(m) = {\tau}_{x}(m)$.
	\end{enumerate}
Define ${L}_{n + 1}$ as follows. For any $x \in {U}_{n}$, ${L}_{n + 1}({\tau}_{x}) = x$. For any $\sigma \in {T}_{\alpha}({l}_{n}) \setminus \{{\sigma}^{\ast}\}$, ${L}_{n + 1}({\tau}_{\sigma}) = {L}_{n}(\sigma)$. This finishes case 1.

	Case II: ${\sigma}_{n} = {\sigma}^{\ast}$. For each $\sigma \in {T}_{\alpha}({l}_{n}) \setminus \{{\sigma}_{n}\}$, let $\langle {p}_{0}(\sigma), \bar{h}(\sigma) \rangle = {L}_{n}(\sigma)$. Let ${i}_{0}(\sigma) < n$ witness (7) for ${L}_{n}(\sigma)$ and let ${j}_{1}(\sigma) < n$ witness (8) for ${L}_{n}(\sigma)$. Let ${U}_{n}$ be the set of all $\langle {p}_{1}, \bar{\hh} \rangle \in {V}_{n}$ such that there is no $\sigma \in {T}_{\alpha}({l}_{n}) \setminus \{{\sigma}_{n}\}$ so that  
	\begin{align*}
		{p}_{0}(\sigma) \subset {p}_{1} \wedge {i}_{0}(\sigma) = {i}_{1} \wedge {j}_{1}(\sigma) < {j}_{2} \wedge \forall i \leq {j}_{1}(\sigma) \[{\hh}_{{\epsilon}_{i}} \restrict \dom{({h}_{{\epsilon}_{i}})} = {h}_{{\epsilon}_{i}}\].
	\end{align*} 
Here ${i}_{1} \leq n$ and ${j}_{2} \leq n$ witness (9) and (10) respectively with respect to $\langle {p}_{1}, \bar{\hh} \rangle$. Choose ${l}_{n + 1} > {l}_{n}$ large enough so that ${I}^{\alpha}_{{k}_{n}} \subset [{l}_{n}, {l}_{n + 1})$ and so that it is possible to choose $\{{\tau}^{\ast}\}$, $\{{\tau}_{x}: x \in {U}_{n}\}$, and $\{{\tau}_{\sigma}: \sigma \in {T}_{\alpha}({l}_{n}) \setminus \{{\sigma}_{n}\} \}$, subsets of ${\omega}^{{l}_{n + 1}}$, satisfying the following conditions.
	\begin{enumerate}
		\item[(15)]
			${\tau}^{\ast} \supset {\sigma}_{n}$. For each $x \in {U}_{n}$, ${\tau}_{x} \supset {\sigma}_{n}$. For each $\sigma \in {T}_{\alpha}({l}_{n}) \setminus \{{\sigma}_{n}\}$, ${\tau}_{\sigma} \supset \sigma$.
		\item[(16)]
			${\tau}^{\ast}$ is the right most branch of ${T}_{\alpha}({l}_{n + 1})$. For each $x \in {U}_{n}$, there exists $m \in [{l}_{n}, {l}_{n + 1})$ such that ${\tau}^{\ast}(m) \neq {\tau}_{x}(m)$. For each $x, y \in {U}_{n}$, if $x \neq y$, then there is $m \in [{l}_{n}, {l}_{n + 1})$ so that ${\tau}_{x}(m) \neq {\tau}_{y}(m)$. For each $x = \langle {p}_{1}, \bar{\hh} \rangle \in {U}_{n}$, if $\{{i}^{\ast}\} = {I}^{\alpha}_{{k}_{n}} \cap \dom{({p}_{1})}$, then ${p}_{1}({i}^{\ast}) = {\tau}_{x}({i}^{\ast})$.
		\item[(17)]
			For each $x \in {U}_{n}$ and $m \in [{l}_{n}, {l}_{n + 1})$, ${\tau}_{x}(m) \neq {\tau}^{\ast}(m)$. For each $\sigma \in {T}_{\alpha}({l}_{n}) \setminus \{{\sigma}_{n}\}$ and for each $m \in [{l}_{n}, {l}_{n + 1})$, ${\tau}^{\ast}(m) \neq {\tau}_{\sigma}(m)$, and for each $x \in {U}_{n}$, ${\tau}_{\sigma}(m) \neq {\tau}_{x}(m)$. For each $\sigma, \eta \in {T}_{\alpha}({l}_{n}) \setminus \{{\sigma}_{n}\}$, if $\sigma \neq \eta$, then for all $m \in [{l}_{n}, {l}_{n + 1})$, ${\tau}_{\sigma}(m) \neq {\tau}_{\eta}(m)$.
		\item[(18)]
			For each $i \leq n$, $\tau \in {T}_{{\epsilon}_{i}}({l}_{n + 1})$, $m \in [{l}_{n}, {l}_{n + 1})$, and $\sigma \in {T}_{\alpha}({l}_{n}) \setminus \{{\sigma}_{n}\}$, ${\tau}^{\ast}(m) \neq \tau(m)$ and ${\tau}_{\sigma}(m) \neq \tau(m)$. For each $x = \langle {p}_{1}, \bar{\hh} \rangle \in {U}_{n}$, $i \leq {j}_{2}$, $\tau \in {\tau}_{{\epsilon}_{i}}({l}_{n + 1})$ and $m \in [{l}_{n}, {l}_{n + 1})$, if ${\tau}_{x}(m) = \tau(m)$, then $m \in \dom{({p}_{1})}$ and ${p}_{1}(m) = {\tau}_{x}(m)$. 
	\end{enumerate}
For each $\sigma \in {T}_{\alpha}({l}_{n}) \setminus \{{\sigma}_{n}\}$, define ${L}_{n + 1}({\tau}_{\sigma}) = {L}_{n}(\sigma)$. For each $x \in {U}_{n}$, set ${L}_{n + 1}({\tau}_{x}) = x$. This completes the construction. We now check that it is as required.
\begin{Lemma} \label{lem:activity}
	For each $f \in \[{T}_{\alpha}\]$, there are infinitely many $n \in \omega$ such that ${\sigma}_{n} = f \restrict {l}_{n}$.
\end{Lemma}
\begin{proof}
	For each $n \in \omega$ put $\Theta(n) = \min{\{\Delta(\sigma): \sigma \in {T}_{\alpha}({l}_{n})\}}$. It is clear from the construction that $\Theta(n + 1) \geq \Theta(n)$. If the lemma fails, then there are $m$ and $\tau \in {T}_{\alpha}({l}_{m + 1})$ with the property that for infinitely many $n > m + 1$, there is a $\sigma \in {T}_{\alpha}({l}_{n})$ such that $\Theta(n) = \Delta(\sigma) = m$ and $\sigma \restrict {l}_{m + 1} = \tau$. Let $\tau$ be the left most node in ${T}_{\alpha}({l}_{m + 1})$ with this property. Choose ${n}_{1} > {n}_{0} > m + 1$ and $\sigma \in {T}_{\alpha}({l}_{{n}_{1}})$ such that $\Theta({n}_{1}) = \Theta({n}_{0}) = \Delta(\sigma) = m$, $\sigma \restrict {l}_{m + 1} = \tau$, and there is no $\eta \in {T}_{\alpha}({l}_{{n}_{0}})$ such that $\Delta(\eta) = m$ and $\eta \restrict {l}_{m + 1}$  is to the left of $\tau$. Note that $\Delta(\sigma \restrict {l}_{{n}_{0}}) = m$. So ${\sigma}_{{n}_{0}}$ is to the left of $\sigma \restrict {l}_{{n}_{0}}$, and ${\sigma}_{{n}_{0}} \restrict {l}_{m + 1}$ is not to the left of $\tau$, whence ${\sigma}_{{n}_{0}} \restrict {l}_{m + 1}= \tau$. But then there is some $n \in [m + 1, {n}_{0})$ where $\sigma \restrict {l}_{n}$ was active, a contradiction.  
\end{proof}
Note that Lemma \ref{lem:activity} implies that for any $\sigma \in {T}_{\alpha}$, there is a unique minimal extension of $\sigma$ which is active.
\begin{Lemma} \label{lem:a.d}
	${T}_{\alpha}$ is finitely branching and ${\bigcup}_{\epsilon \leq \alpha} {[{T}_{\epsilon}]}$ is a.d.\ in $\BS$.
\end{Lemma}
\begin{proof}
	It is clear from the construction that ${T}_{\alpha}$ is finitely branching. Fix $f, g \in [{T}_{\alpha}]$, with $f \neq g$. Let $n = \max\{i \in \omega: f \restrict {l}_{i} = g \restrict {l}_{i}\}$. It is clear from the construction that $\forall m \geq {l}_{n + 1}\[f(m) \neq g(m)\]$.
	
	Next, fix $\epsilon < \alpha$. Suppose $\epsilon = {\epsilon}_{i}$. Let $h \in \[{T}_{{\epsilon}_{i}}\]$ and $f \in \[{T}_{\alpha}\]$, and suppose for a contradiction that $\lc h \cap f\rc = {\aleph}_{0}$. So there are infinitely many $n \in \omega$ such that $f \restrict [{l}_{n}, {l}_{n + 1})  \cap  h \restrict [{l}_{n}, {l}_{n + 1}) \neq = 0$. For any $n \geq i$, this can only happen if $f \restrict {l}_{n} = {\sigma}_{n}$ and $f \restrict {l}_{n + 1} = {\tau}_{{x}_{n}}$ for some ${x}_{n} \in {U}_{n}$. Put ${x}_{n} = \langle {p}_{1, n}, {\bar{\hh}}_{n} \rangle$. Note that in this case ${L}_{n + 1}(f \restrict {l}_{n + 1}) = {x}_{n}$. For such $n$, let ${j}_{2}(n)$ be as in (10) with respect to ${x}_{n}$. So for infinitely many such $n$, ${j}_{2}(n) \geq i$. But then for infinitely many such $n$, ${\hh}_{{\epsilon}_{i}, n}(h \restrict \max{(\dom{({p}_{1, n})})} + 1) < {\hh}_{{\epsilon}_{i}, n}(h \restrict {l}_{n})$, producing an infinite strictly descending sequence of ordinals.    
\end{proof}
\begin{Lemma} \label{lem:vd}
	For each $A \in \cube$ and $p: A \rightarrow \omega$, there are $\alpha < {\omega}_{1}$ and $f \in \[{T}_{\alpha}\]$ such that $\lc p \cap f \rc = {\aleph}_{0}$.
\end{Lemma}
\begin{proof}
	Suppose for a contradiction that there are $A \in \cube$ and $p: A \rightarrow \omega$ such that $p$ is a.d.\ from $\[{T}_{\alpha}\]$, for each $\alpha < {\omega}_{1}$. Let $M \prec H(\theta)$ be a countable elementary submodel containing everything relevant. Put $\alpha = M \cap {\omega}_{1}$. For each $\epsilon < \alpha$, let ${H}_{\epsilon}$ denote ${H}_{{T}_{\epsilon}, p}$, and note that ${H}_{\epsilon}$ and $\ran{({H}_{\epsilon})}$ are members of $M$. Let ${\xi}_{\epsilon} = \sup{(\ran{({H}_{\epsilon})})} + 1 < \alpha$. Find $g \in M \cap \BS$ such that for $n \in \omega$, ${H}_{\epsilon}''{{T}_{\epsilon} \restrict n} \subset \{{e}_{{\xi}_{\epsilon}}(j): j \leq g(n)\}$. Since $\forallbutfin n \in \omega \[g(n) \leq {f}_{\alpha}(n)\]$, it follows from (4) that for all but finitely many $n \in \omega$, for all $\sigma \in {T}_{\epsilon} \restrict n$, ${H}_{\epsilon}(\sigma) \in {w}_{\alpha}(n)$. Now, find $q \subset p$ such that $\forall m \in \dom{(q)}\[q(m) \leq {f}_{\alpha}(m)\]$ and $\forallbutfin n \in \omega \[\lc \dom(q) \cap {I}^{\alpha}_{n} \rc = 1 \]$. Note that for any $\epsilon < \alpha$, (${\ast}_{1}$) and (${\ast}_{2}$) are satisfied when $T$ is replaced there with ${T}_{\epsilon}$, ${H}_{T, p}$ is replaced with ${H}_{\epsilon}$, $A$ with $\dom{(q)}$, and $p$ with $q$. But now, it follows from the construction that there is $f \in \[{T}_{\alpha}\]$ such that for infinitely many $n \in \omega$, there is $m \in [{l}_{n}, {l}_{n + 1}) \cap \dom{(q)}$ such that $q(m) = f(m)$.  
\end{proof}
\section{Remarks and Questions} \label{sec:quest}
	The construction in this paper is very specific to ${\omega}_{1}$; indeed, it is possible to show that $\d$ is not always an upper bound for $\ac$. A modification of the methods of Section 4 of \cite{sh700} shows that if $\kappa$ is a measurable cardinal and if $\lambda = \cf{(\lambda)} = {\lambda}^{\kappa} > \mu = \cf{(\mu)} > \kappa$, then there is a c.c.c.\ poset $\P$ such that $\lc \P \rc = \lambda$, and $\P$ forces that $\b = \d = \mu$ and $\a = \ac = \c = \lambda$.

	As mentioned in Section \ref{sec:intro}, we see the result in this paper as providing a weak positive answer to the following basic question, which has remained open for long.
\begin{Question} \label{q:aandd}
	If $\d = {\aleph}_{1}$, then is $\a = {\aleph}_{1}$?
\end{Question}
There are also several open questions about upper and lower bounds for $\ac$.
\begin{Question} [Brendle and Khomskii~\cite{perfectmad}]\label{q:yj}
	If $\s = {\aleph}_{1}$, then is $\ac = {\aleph}_{1}$?
\end{Question}
\begin{Question}\label{q:h}
	Is  $\h \leq \ac$?
\end{Question}
Regarding Question \ref{q:yj}, it is proved in Brendle and Khomskii~\cite{perfectmad} that if $\V$ is any ground model satisfying $\CH$ and $\P$ is any poset forcing that all splitting families in $\V$ remain splitting families in $\VG$, then $\P$ also forces that $\ac = {\aleph}_{1}$. This result suggests that Question \ref{q:yj} should have a positive answer, and showing this would be an improvement of the result in this paper.
\bibliographystyle{amsplain}
\bibliography{Bibliography}
\end{document}